\documentclass[letterpaper,11pt]{article}

\usepackage{paralist, amsmath, amsthm, amsfonts, amssymb, xy,}
\usepackage[margin=1in]{geometry} 
\usepackage{hyperref}
\usepackage[T1]{fontenc}
\usepackage{calligra}
\usepackage{stackrel}

\input xy
\xyoption{all}

\newcommand{\NN}{\mathbb{N}}
\newcommand{\ZZ}{\mathbb{Z}}
\newcommand{\SM}{\mathbf{Sm}_k}
\newcommand{\SCH}{\mathbf{Sch}_k}

\newcommand{\srarrow}{\twoheadrightarrow}
\newcommand{\irarrow}{\hookrightarrow}

\newcommand{\flag}{\mathcal{F}\ell\,}
\newcommand{\flagC}{\flag^{\text{\tiny $C$}}}

\newcommand{\Proj}{\mathbb{P}}
\newcommand{\Laz}{\mathbb{L}}
\newcommand{\trecd}{\cdot\cdot\cdot}
\newcommand{\tred}{\ldots}
\newcommand{\unddot}{_\textbf{\textbullet}}

\newcommand{\spec}{{\rm Spec\,}}

\newtheorem{theorem}{Theorem}[section]
\newtheorem{lemma}[theorem]{Lemma}
\newtheorem{proposition}[theorem]{Proposition}
\newtheorem{corollary}[theorem]{Corollary}
\newtheorem{definition}[theorem]{Definition}

\theoremstyle{definition} \newtheorem{remark}[theorem]{Remark}
\theoremstyle{definition} 
\theoremstyle{definition} 
\theoremstyle{definition} 
\date{}

\title{\textbf{Generalised symplectic Schubert classes}}

\author{THOMAS  HUDSON}
\begin{document}

\maketitle
\begin{abstract}
Under the assumption that the base field $k$ has characteristic 0, we compute the algebraic cobordism fundamental classes of a family of Schubert varieties isomorphic to full and symplectic flag bundles. We use this computation to prove a formula for the push-forward class of the Bott-Samelson resolutions associated to the symplectic flag bundle. We specialise our formula to connective $K$-theory and the Grothendieck ring of vector bundles, hence obtaining a description for all Schubert classes. 

\vspace*{1\baselineskip}
 
\end{abstract}

\tableofcontents

\section{Introduction} 

The question which represented the starting point of this work can be phrased as follows: ``given a symplectic flag bundle $\flagC V\rightarrow X$ with a full isotropic flag $W\unddot:=(W_1\irarrow \tred \irarrow W_n)\irarrow V$, how can one express the algebraic cobordism fundamental class of $\Omega_{w_0}$, its smallest Schubert variety?''.

 In order to provide some context and explain the relevance of the issue, it can be convenient to mention the work of Fulton on a closely related subject. In \cite{SchubertclassicalFulton} he considered the problem of describing the Schubert classes $[\Omega_w]_{CH}$ of both symplectic and orthogonal flag bundles and he gave a recursive description which only required him to identify the explicit description of the class associated to $\Omega_{w_0}$, the simplest of all Schubert varieties. All others were recovered from this special element through the use of $\partial_i$, the so-called divided difference operators. At the level of geometry the construction relies on a family of desingularizations $R_I\stackrel{r_I}\longrightarrow \flagC V$ known as Bott-Samelson resolutions and, by design, the quantities that are actually being computed are the push-forwards $\mathcal{R}_I:=r_{I*}[R_I]_{CH}$, which one proves to coincide with the Schubert classes.
 
This approach had previously been used in \cite{FlagsFulton}, where the Schubert classes of the full flag bundle $\flag V$ were expressed by means of the double Schubert polynomials of Lascoux and Sch\"{u}tzenberger. It is interesting to observe that in this case Fulton's computation of $[\Omega_{w_0}]_{CH}$ proved to be extremely general: in \cite{PieriFulton} Fulton and Lascoux successfully adapted it to the Grothendieck ring of vector bundles $K^0$ and more recently, in \cite{ThomHudson}, both proofs have been generalised to algebraic cobordism.

Algebraic cobordism, denoted $\Omega^*$, was constructed in \cite{AlgebraicLevine} by Levine and Morel as an algebro-geometric analogue of Quillen's complex cobordism $MU^*$. Our interested in this specific example of oriented cohomology theory is motivated by its universality. For any other such theory $A^*$ there exists a unique natural transformation $\vartheta_A: \Omega^* \rightarrow A^*$ preserving all given structures: pull-back and push-forward morphisms as well as Chern classes. In particular this implies that all the formulas which are valid in $\Omega^*$ can be specialised to other oriented cohomology theories such as, for instance, $CH^*$ and $K^0[\beta,\beta^{-1}]$, a graded version of $K^0$. 

In order to appreciate in which ways a general $A^*$ may differ from $CH^*$, it can be convenient to consider the extra freedom allowed to Chern classes. In the Chow ring the first Chern class behaves linearly with respect to the tensor product of line bundles, while in general this is no longer the case. Instead, there exist unique power series $F_A\in A^*(k)[[u,v]]$ and $\chi_A\in A^*(k)[[u]]$ such that, for any choice of vector bundles $L,M\rightarrow X$, one has
$$c_1^A(L\otimes M)=F_A\Big(c^A_1(L),c^A_1(M)\Big) \quad\text{ and }\quad c_1^A(L^\vee)=\chi_A\big(c_1^A(L)\big).$$ 
These two series are usually referred to as the formal group law and the formal inverse associated to $A^*$ and they play a key role in expressing $[\Omega_{w_0}]_\Omega$. It is worth pointing out that the universality of algebraic cobordism is also reflected at this level. In fact, Levine and Morel show that the pair $(\Omega^*(k),F_\Omega)$ is itself universal among formal group laws, with the coefficient ring being isomorphic to the Lazard ring $\Laz$. 

In order to stress the differences with the generalised setting, let us now say a few words about Fulton's computation of the initial class. From the geometric point of view, his approach in dealing with the symplectic case aims at reducing the problem to a situation in which $\flagC V$ is replaced by a full flag bundle, more specifically $\flag W_n$. He performed this by identifying the fundamental class of the locus over which $U_n$, the largest bundle in the universal isotropic flag $U\unddot$, coincides with $W_n$. Since for full flag bundles the formula was already known, he simply had to multiply together the two answers.
    
Although theoretically one could follow the same strategy, the identification of the first of the two classes requires one to perform explicit manipulations of certain Chern polynomials, something extremely difficult to do in cobordism, given the increase in complexity due to the formal group law. As an alternative, one can reverse the procedure and start from the opposite end, considering the locus where $U_1$  is equal to $W_1$. Not only the resulting subscheme happens to be a Schubert variety of $\flagC V$, it is itself isomorphic to a symplectic flag bundle of smaller size. As a result the construction can be applied again and it is possible to make use of inductive methods to compute the fundamental classes of the elements of this nested family of Schubert variety, among which there is also $\Omega_{w_0}$.


\begin{theorem} \label{thm SFlag}
Let $V$ be a vector bundle of rank $2n$ over $X\in \SM$, endowed with an everywhere nondegenerate skew-symmetric form and let $W\unddot\subset V$ be a full isotropic flag of subbundles. Let moreover $U\unddot$ denote the universal isotropic flag of $\flagC V$. Fix a nonnegative integer $m< n$ and set $V_m:=W_{n-m}^\perp/W_{n-m}$. As an element of $\Omega^*(\flagC V)$ the fundamental class of the Schubert variety $\flagC V_m$ is given by
$$[\flagC V_m]_{\Omega}=\prod_{\substack{i+j\leq n\\1\leq j\leq n-m}} F_\Omega\Big(x_i, \chi_\Omega(y_j)\Big) \cdot
\prod_{\substack{i+j\leq n+1\\1\leq j\leq n-m}} F_\Omega\Big(\chi_\Omega(x_i), \chi_\Omega(y_j)\Big)$$
where $x_i=c_1\big(U_{n+1-i}/U_{n-i}\big)$ and $y_i:=c_1\big(W_i/W_{i-1}\big)$.
\end{theorem}

It is interesting to observe that the same procedure can be applied to full flag bundles. In this case one is able to recover and improve \cite[Proposition 4.5]{ThomHudson}, which only provided the fundamental class of $\flag V_1\simeq \Omega_{w_0}$.

\begin{theorem}\label{thm Flag}
Let $V$ be a vector bundle of rank $n$ over $X\in \SM$ and $W\unddot$ a full flag of subbundles. Fix a positive integer $m\leq n$, set $V_m:=V/W_{n-m}$ and denote by $U\unddot$ the universal flag of subbundles of $\flag V$. As an element of $\Omega^*(\flag V)$ the fundamental class of the Schubert variety $\flag V_m$ is given by
$$[\flag V_m]_{\Omega}=\prod_{\substack{i+j\leq n\\1\leq j\leq n-m}} F_\Omega\Big(x_i, \chi_\Omega(y_j))\Big)$$
where  $x_i:=c_1\big(U_{n+1-i}/U_{n-i}\big)$ and $y_i:=c_1\big(W_i/W_{i-1}\big)$.
 \end{theorem}

At this point the reader might wonder whether there exists a generalization to $\Omega^*$ of the divided difference operators of the Chow ring, allowing one to compute the push-forward classes $\mathcal{R}_I^\Omega$. This is indeed the case, in fact one has operators $\overline{A_i}: \Omega^*(\flagC V)\rightarrow\Omega^*(\flagC V)$ given by
$$\overline{A_0}(f):=(1+\sigma_0)\frac{f}{F_\Omega\Big(\chi_\Omega(x_1),\chi_\Omega(x_1)\Big)}\quad
\text{ and } \quad\overline{A_i}(f):=(1+\sigma_i)\frac{f}{F_\Omega\Big(x_i,\chi_\Omega(x_{i+1})\Big)},$$
where $i\in\{1,\tred,n-1\}$ and the operators $\sigma_0$ and $\sigma_i$ respectively exchange $x_1$ with $\chi_\Omega(x_1)$ and $x_i$ with $x_{i+1}$. These generalised divided difference operators were first introduced in \cite{SchubertBressler} by Bressler and Evens to study the Schubert calculus of the complex cobordism of the flag manifold and later, in \cite{SchubertHornbostel}, they were trasferred to algebraic cobordism by Hornbostel and Kiritchenko. Through the use of these operators one obtains the following symplectic analogue of \cite[Theorem 4.8]{ThomHudson}.

\begin{theorem}
Under the hypothesis and notations of theorem \ref{thm SFlag}, let $I=(i_1,...,i_l)$ be any tuple  with $i_j\in\{0,1,\tred,n-1 \}$. As an element of $\Omega^*(\flagC V)$ the push-forward class associated to  $R_I\stackrel{r_I}\rightarrow \flagC V$ is given by
$$\mathcal{R}_I=\overline{A_{i_l}}\trecd \overline{A_{i_1}}[\Omega_{w_0}]_\Omega
\quad\text{ with }\quad [\Omega_{w_0}]_\Omega=
\prod_{i+j\leq n} F_\Omega\Big(x_i, \chi_\Omega(y_j))\Big)  \cdot
 \prod_{i+j\leq n+1} F_\Omega\Big( \chi_\Omega(x_i),  \chi_\Omega(y_j)\Big).$$
\end{theorem}

Unfortunately, for a general oriented cohomology theory the overall situation is not as straightforward as in the Chow ring. First of all, from the general theory not all Schubert varieties have a well defined fundamental class, only those that are l.c.i. schemes. This in fact presents an interesting challenge, which element of $\Omega^*(\flag V)$ (or of $\Omega^*(\flagC V)$) should one attach to a given Schubert cell, given that the cellular decomposition of the full (or symplectic) flag bundle still garantees the existence of an additive basis?

 One may expect, again by taking $CH^*$ as a model, that it suffices to consider the push-forward of the fundamental class of any desingularization. However in this case the result may depend on the choice made. For instance, unlike what happens with the regular divided difference operators, the $\overline{A_i}$'s do not satisfy the braid relations and as a consequence different $R_I$'s associated to the same Schubert variety $\Omega_w$ may produce different classes. Moreover, in the cases in which it is defined, these classes may even all differ from the natural choice $[\Omega_w]_\Omega$. This, perhaps, might convince the reader of the value of having at our disposal, together with the classes $\mathcal{R}_I^\Omega$, also the closed formulas of theorems \ref{thm SFlag} and \ref{thm Flag}.

All these difficulties can be avoided in case one decides to consider connective $K$-theory, denoted $CK^*$, the most general theory for which the braid relations still hold for the associated divided difference operators $\phi_i$.  In particular, it is possible to associate to every element $w$ of $\mathbf{W}_n$, the Weyl group parametrizing the Schubert varieties of $\flagC V$, a well defined operator $\phi_w$. As a functor $CK^*$ is defined as $\Omega^*\otimes_\Laz\ZZ[\beta]$ and its associated formal group law and formal inverse are given as follows: 
$$u\oplus v:=F_{CK}(u,v)=u+v-\beta\cdot uv \quad  \text{ and } \quad  \ominus u:=\chi_{CK}(u)=\frac{-u}{1-\beta u}$$  
 In this theory, which generalises both $CH^*$ and $K^0$, fundamental classes happen to be well defined for all Schubert varieties and one can obtain the following generalization of Fulton's result.

\begin{corollary}
Under the hypothesis and notations of theorem \ref{thm SFlag},  for every $w\in\mathbf{W}_n$ the $CK^*$ fundamental class of the Schubert variety $\Omega_w$ is given by
$$[\Omega_w]_{CK}=\overline{\phi_{w_0w}}\Big([\Omega_{w_0}]_{CK} \Big)
\quad \text{ where }\quad
[\Omega_{w_0}]_{CK}=\prod_{i+j\leq n} x_i\ominus y_j  \cdot
 \prod_{i+j\leq n+1} \ominus x_i \ominus y_j.$$

 Furthermore, if one sets $\beta$ equals to 0 and 1, the previous formula respectively provides a description of $[\Omega_w]_{CH}$ and of $[\mathcal{O}_{\Omega_w}]_{K^0}$.     
\end{corollary}



The internal organisation of this work is as follow. Section 2 contains some basic definitions and notions related to oriented cohomology theories, algebraic cobordism and connective $K$-theory. All the background material related to full flags bundles and symplectic flag bundles is contained in section 3, together with the actual computations of the fundamental classes.  At first we focus our attention on the full flag bundle, describing the situation in detail, and then we explain what needs to be modified in order to be able to deal with the symplectic case. Finally, in section 4 we introduce Bott-Samelson resolution, we compute the classes $\mathcal{R}_I^A$ for any oriented cohomology theory and we specialise this result to connective $K$-theory as well as to $CH^*$ and $K^0$.

\paragraph*{Acknowledgements: }

The early steps of this work took place while I was at The University of Nottingham and it further developed during my time at KAIST. I would like to thank both institutions for the excellent working conditions and in particular Alexander Vishik and Jinhyun Park for their support and encouragement. I would also like to thank Marc Levine and Jerzy Weyman for useful discussions. 

The support of the National Research Foundation of Korea (NRF) through the grants funded by the Korea government (MSIP) (No. 2013-042157, 2014-001824 and 2011-0030044) and of the Engineering and Physical Sciences Research Council (EPSRC) through the Responsive Mode grant EP/G032556/1
is gratefully acknowledged.

\paragraph*{Notations and conventions:}

Given a field $k$ of characteristic 0, we will write $\SCH$ to denote the category of separated schemes of finite type over $\spec k$ and $\SCH'$ for its subcategory obtained by considering only projective morphisms. $\SM$ will represent the full subcategory of $\SCH$ consisting of schemes smooth and quasi-projective over $\spec k$. In general by smooth morphism we will always mean smooth and quasi-projective.

\section{Preliminaries}\label{section prem}

The aim of this section is to provide a brief introduction to oriented cohomology theories and in particular to algebraic cobordism, setting up the notation needed in order to make use of its universality. For an exhaustive treatment we refer the reader to \cite[Chapter 1]{AlgebraicLevine}.  We will also present a modification of the projective bundle formula for the special case of vector bundles endowed with a non-degenerate skew-symmetric form.
 
\subsection{Oriented cohomology theories}

An oriented cohomology theory (OCT) cosists of a contravariant functor from $\SM$ to the category of graded abelian rings, together with a family of push-forwards maps $f_*$ associated to every projective morphism $f$. This collection of data is supposed to satisfy some functorial compatibilities as well as two properties of geometric nature: the projective bundle formula and the extended homotopy property.
 One well-known example of such objects is represented by the Chow ring $CH^*$ and, as it was mentioned in the introduction, it can be convenient to look at the differences in the behaviour of the Chern classes. In view of the projective bundle formula, every OCT $A^*$ is endowed with a theory of Chern classes and as a consequence, for any given vector bundle $E$ over a smooth scheme $X$ of rank $n$, one has classes $c^A_1(E),\tred, c^A_n(E)\in A^*(X)$. In the Chow ring the first Chern class is a homomorphism for the tensor product of line bundles, hence for every pair $L,M\rightarrow X$ one has:
$$c_1^{CH}(L\otimes M)=c_1^{CH}(L)+ c_1^{CH}(M).$$
Although in general this is no longer true for OCTs, it is still possible to provide a description of $c_1^A(L\otimes M)$ in terms of the first Chern classes of the factors and this can be achieved uniformly. In fact, there exists a unique power series $F_A\in A^*(k)[[u,v]]$ such that the equality
$$c_1^A(L\otimes M)=F_A\left(c_1^A(L),c_1^A(M)\right)$$         
is satisfied in $A^*(X)$ for any choice of $X$ and of the line bundles. It turns out that the pair $(A^*(k),F_A)$ constitutes a commutative formal group law of rank 1 and, as a consequence, there also exists a unique power series $\chi_{F_A}\in A^*(k)[[u]]$, known as formal inverse, such that 
$$F_A\left(u, \chi_{F_A}(u) \right)=0.$$
From the point of view of Chern classes $\chi_A$ can be thought as the algebraic counterpart of taking the dual of a line bundle: for every $L$ one has $c_1^A(L^\vee)=\chi_{F_A}\left(c_1^A(L)\right).$

\begin{remark}\label{rmk pol}
It should be stressed that, although $F_A$ and $\chi_A$ are defined as series, on a given scheme they can be actually be considered as polynomials. In fact, for every family of line bundles $(L_1,\tred,L_m)$ the product 
$c^A_1(L_1)\trecd c^A_1(L_m)$ vanishes if $m>\text{dim}_k\, X$, therefore it is possible to disegard the monomials whose total degree exceeds the dimension of $X$.
\end{remark}  

Let us now provide two basic examples of formal group laws. For any given ring $R$ the simplest example is represented by the additive formal group law, given by
 $$F_a(u,v)=u+v \quad \text{ and }\quad\chi_{F_a}(u)=-u.$$
 An OCT whose formal group law is of this type is also referred to as additive. Slightly more complex are the multiplicative formal group laws, which require the choice of a parameter $b\in R$. One then sets
$$F_m(u,v)=u+v-b\cdot uv \quad \text{ and }\quad\chi_{F_m}(u)=\frac{-u}{1-bu}.$$ 
In case $b$ happens to be an invertible element, then $(R,F_m)$ is said to be periodic. An example of this type of theory can be easily constructed by adding a grading to $K^0$, the Grothendieck ring of vector bundles. To achieve this one tensors it with $\ZZ[\beta,\beta^{-1}]$, sets $\text{deg}\, \beta={-1}$ and modifies the usual expressions for the push-fowards and the pull-backs as follows:  
  $$f^*([E]\cdot\beta^m)=[f^*(E)]\cdot \beta^m\hspace{0.05 mm}\quad \text{ and }\quad
g_*([E]\cdot\beta^m)= \sum_{i=0}^{\infty} (-1)^i [R^i g_*(\mathcal{E})]\cdot \beta^{m-d}.$$
Notice that the definition of the push-forward morphism recquires us to identify the vector bundle $E$ with its corresponding locally free sheaf $\mathcal{E}$. We will write $K^0[\beta,\beta^{-1}]$ to refer to the resulting OCT, which happens to be periodic as well as multiplicative with $b=\beta$. To make things more concrete, let us add that for any line bundle $L$ we have $c_1(L):=(1-[L^\vee]_{K^0})\beta^{-1}$.

Another fundamental example of formal group law, at the opposite side of the spectrum from the point of view of complexity, was constructed in \cite{Lazard} by Lazard. It is defined over a polynomial ring with inifinitely many variables $\Laz:=\ZZ[x_i]_{i>0}$ (now known as the Lazard ring) and it is given by 
$$F(u,v)=u+v+\sum_{i,j\geq 1}a_{i,j} u^i v^j,$$
where $a_{1,1}$ is canonically identified with $x_1$ and all other coefficients $a_{i,j}$ can be expressed as polynomials in the variables up to index $i+j-1$. The pair $(\Laz,F)$ is usually referred to as the universal formal group law because any other $(R,F_R)$ is classified by a unique ring  homomorphism  $\varPhi_{(R,F_R)}:\Laz\rightarrow R$, mapping the coefficients of $F$ to those of $F_R$. For any OCT $A^*$ we will write  $\varPhi_A$ for the morphism classifying the corresponding formal group law. As an example we can consider $\varPhi_{CH}:\Laz \rightarrow \ZZ$, which, as all homomorphisms associated to additive laws, maps all the variables to 0. On the other hand multiplicative formal group laws, as for instance $K^0[\beta,\beta^{-1}]$, are classified by morphisms which may differ from the previous ones only for the value at $x_1$, which is $-b$.

It was Quillen the first to notice the link between the behaviour of the first Chern class and formal group laws. In fact, in \cite{ElementaryQuillen} he introduced in the topological setting the notion of OCT as a functor defined over the category of differentiable manifolds, characterizing $MU^*$ as the universal such theory.  He also proved that the homomorphism classifying $F_{MU}$ is actually an isomorphism, hence showing that $MU^*(k)$ can be identified with the Lazard ring. It was in view of these results that Levine and Morel decided to extend the definition of OCTs to smooth schemes: their goal was to identify the algebro-geometric counterpart of $MU^*$. The outcome of their efforts was the construction of algebraic cobordism, for which they established the exact analogues of Quillen's results: the coincidence of the coefficient ring with $\Laz$ and the universality of $\Omega^*$ among OCTs. More specifically, they prove that for any other theory $A^*$, there exists a unique morphism of OCTs $\theta_A:\Omega^*\rightarrow A^*$, i.e. a natural transformation which is also compatible with the family of push-forwards.   

It formally follows from this result that, starting from any $(R,F_R)$, it is possible to construct an OCT which is universal among those with such formal group law. One simply has to consider $\Omega^*_{(R,F_R)}:=\Omega^*\otimes_\Laz R$, where the tensor product is taken along $\varPhi_{(R,F_R)}$. In particular this procedure can be applied to $(A^*(k), F_A)$, in which case one may ask whether or not 
$$\theta^{(A^*(k), F_A)}_A:\Omega^*_{(A^*(k), F_A)}\rightarrow A^*$$
 is an isomorphism.  By adopting this point of view Levine and Morel succeeded in characterizing $CH^*$ and $K^0[\beta,\beta^{-1}]$ as the universal theories for the additive and multiplicative periodic laws.   

There is a second aspect, together with the formal group law, which makes OCTs in general (and $\Omega^*$ in particular) more difficult to handle than the Chow ring and the Grothendieck ring of vector bundles. In these two cases one has well defined fundamental classes for all equidimensional schemes. In fact, for a $d$-dimensional subscheme $X\stackrel{i}\rightarrow Y$ one can set
\begin{eqnarray}\label{eqn fund}
[X]_{CH}=\sum_{i=1}^n m_i[X_i] \quad \text{ and } \quad 
[X]_{K^0[\beta,\beta^{-1}]}=[\mathcal{O}_X]_{K^0}\cdot \beta^{d},
\end{eqnarray} 
where $X_1,\tred,X_n$ are the irreducible components of $X$, the coefficient $m_i$ represents the length of the  local ring $\mathcal{O}_{X,X_i}$ and $[\mathcal{O}_X]_{K^0}$ is the element corresponding to the structure sheaf of $X$ under the identification of $K^0(Y)$ with $G_0(Y)$, the Grothendieck group of quasi-coherent sheaves. 

In general, in order to define fundamental classes, it is necessary to replace the notion of OCT with that of oriented Borel-Moore homology. In this context one does not require the existence of pull-backs for all morphisms, but only for l.c.i. morphisms. This has two immediate consequences, it is possible to define $\Omega_*$ on a larger category (for instance $\SCH'$) and the definition of the fundamental class can be extended from $\SM$ to l.c.i. schemes by making use of the pull-back along the structural morphism. More precisely, in $\Omega_*(X)$ one sets 
$$[X]_\Omega:=\pi^*_X(1) $$
for $1\in\Laz$,
while in $\Omega^*(Y)$ we define the fundamental class to be the push-forward $i_*([X]_\Omega)$. This definition, which is compatible with l.c.i. pull-backs, can be transferred to all other OCTs by making use of the universality of $\Omega^*$ and in this sense it agrees with the expressions in (\ref{eqn fund}). 

Another example of OCT for which it is possible to extend the general definition is connective $K$-theory (denoted $CK^*$), the universal theory associated to the multiplicative formal group law $(\ZZ[\beta],F_m)$ with $b=\beta$. In \cite[Corollary 6.4]{ConnectiveDai}, Dai and Levine proved that for any equidimensional scheme the top components of the oriented Borel-Moore homology theories associated to $CK^*$ and $K^0[\beta,\beta^{-1}]$ are isomorphic and that this isomorphism is compatible with l.c.i. pull-backs. In this way they were able to make use of the second definition in (\ref{eqn fund}) to enlarge the one inherited from $\Omega^*$. Moreover, this assignment is compatible with $\theta^{CK}_{CH}$ and $\theta^{CK}_{K^0[\beta,\beta^{-1}]}$, the morphisms of OCTs arising from the universality of connective $K$-theory. As a direct consequence of this, one has the following lemma. For a proof see \cite[Lemma 2.2]{ThomHudson}.

\begin{lemma}  \label{lem class}
Let $X\in\SCH$ be equi-dimensional and with at worst rational singularities. Let $f:Y\rightarrow X$ be a resolution of singularities of $X$ and denote by $i_X$ the embedding of $X$ into $Z\in \SM$. Then $(i_X\circ f)_*([Y]_{CK})=[X]_{CK}$ as elements of $CK^*(Z)$.  
\end{lemma}
  

\subsection{A symplectic projective bundle formula}

We now present an adaptation of the projective bundle formula to the special case in which the bundle is endowed with a non-degenerate skew-symmetric form. For this, let us first introduce a piece of notation which will allow us to express the statement of the orginal result in a form more suitable for our needs.

For a vector bundle $E\rightarrow X$ of rank $n$ and $l\in\NN$, define $P_l^E\in\Omega^*(X)[z_1,z_2,\dots]$ to be the following symmetric function: $$P_l^E(\mathbf{z}):=\sum_{j=0}^l (-1)^j  c_{l-j}(E)h_j(\mathbf{z}).$$
\begin{theorem}[Projective bundle formula]\label{th Proj}
For $X\in \SM$, let $\pi :E\rightarrow X$ be a vector bundle of rank $n$ and consider the $\Omega^*(X)$-algebra homorphism 
$$\Omega^*(X)[\xi]\stackrel{\varphi}\rightarrow \Omega^*(\Proj (E))$$
which maps $\xi$ to $c_1(O(1))$. Then $\varphi$ is an epimorphism and furthermore ${\rm Ker}\,\varphi=\big(P_n^E(\xi)\big)$. 
\end{theorem}

The observation which suggests this reformulation is that $P_n^E$ can be viewed as the polynomial describing $c_n(\text{Ker\,}(\pi^*E\srarrow O(1)))$ in terms of $c_1(O(1))$. Since the kernel has rank $n-1$, this Chern class vanishes and a result $P_n^E(\xi)$ is mapped to 0. The projective bundle formula can be rephrased by saying that all other relations are a consequence of this one. Notice that, at the level of Chern polynomials, the observation follows by applying the Whitney formula to the short exact sequence associated to $\pi^*E\srarrow O(1)$. In the special case we are interested in, the duality provides us with more information and we can consider the sequences arising from $\pi^* E\srarrow O(1)^\perp\srarrow O(1)$. In order to make use of this extra piece of information, we need an elementary lemma on symmetric functions and a recursive formula relating polynomials with different indices. 

\begin{lemma} \label{lem symm}
In the ring of symmetric functions in variables $z_1$ and $z_2$ for any $i>0$ we have the following equalities:
\begin{eqnarray}\label{eq symm}
i)\  h_i= e_1\cdot h_{i-1}- e_2\cdot h_{i-2}\ ;\ ii)\ z_1^{i} =h_i-z_2\cdot h_{i-1}.
\end{eqnarray}
\end{lemma}

\begin{proof}
The first equality is nothing but a specialization of the well known formula $e_{-t}h_t=1$, which summarises the relations existing between the power series generated by complete and elementary symmetric functions. The second one easily follows by induction once one has observed that \textit{i)} gives
$$h_i-z_2\cdot h_{i-1}=z_1(h_{i-1}-z_2 \cdot h_{i-2})\ . \qedhere$$ 
\end{proof}

\begin{lemma}\label{lem id}
Let $E$ be a vector bundle  of rank $n$. For any $m\in\NN$ the polynomial $P_m^E(z_1,z_2)$ is given by the following recursive relation:
\begin{align}\label{eq rec}
P_m^E=c_m(E)-e_1(z_1,z_2)\cdot P_{m-1}^E -e_2(z_1,z_2)\cdot P_{m-2}^E.
\end{align}
\end{lemma}
\begin{proof}
Through the use of the definition of $P_j^E$ and some manipulations, from the right hand side we obtain:
$$ c_m(E)- e_1\sum_{i=0}^{m-1} (-1)^i c_{m-1-i}(E)\cdot h_i
-e_2\sum_{i=0}^{m-2} (-1)^i c_{m-2-i}(E)\cdot h_i=  \\$$
$$=c_m(E)-c_{m-1}(E)\cdot e_1+\sum_{i=0}^{m-2} (-1)^i c_{m-2-i}(E) [e_1\cdot h_{i+1}-e_2\cdot h_i].$$
It now suffices to apply lemma \ref{lem symm}, observe that $e_1=h_1$ and combine together the three resulting summands.
\end{proof}

At this stage we can express the Chern classes appearing in the decomposition of the Chern polynomial of a bundle which has two factors dual to each other. Let us remind the reader that, following remark \ref{rmk pol}, we will interpret $\chi$ as a polynomial and not as a power series.

\begin{proposition} \label{prop geom}
Let $E$ be a vector bundle of rank $2n$ over $X\in \SM$, endowed with an everywhere nondegenerate skew-symmetric form $\langle\ ,\ \rangle$. Let $L\irarrow E$ be a line bundle, $L^\perp$ its orthogonal and set $\alpha:=c_1(L)$. Then in $\Omega^*(X)$ for every $m\in\NN$ we have
$$c_m(L^\perp/ L)=P_m^E\big( \alpha,\chi(\alpha) \big).$$
\end{proposition}

\begin{proof}
 Since the form is skew-symmetric one can conclude that $E/L^\perp$ is isomorphic to $L^\vee$ and that $L$ is contained in $L^\perp$. Through the use of the Whitney formula and the expansion of the resulting expression one obtains the following equalities of Chern polynomials:
$$c_t(E)=c_t( L^\perp/L)c_t(L)c_t(L^\vee)=
c_t( L^\perp/L)\Big[1+e_1\big(\alpha,\xi(\alpha)\big)\cdot t+e_2\big(\alpha,\xi(\alpha)\big)\cdot t^2\Big]. $$ 
We argue by induction. For $m=1$ it is sufficient to compare the coefficients of $t$ on both sides of the equality. For general $m$ one first uses the inductive hypothesis to substitute $c_{m-1}( L^\perp/L)$ and $c_{m-2}( L^\perp/L)$ and then applies lemma \ref{lem id}:
$$c_m(L^\perp/L)= c_m(E)- e_1\cdot P_{m-1}^E\big(\alpha, \chi(\alpha)\big)
-e_2\cdot P_{m-2}^E\big(\alpha, \chi(\alpha)\big)=P_m^E\big(\alpha, \chi(\alpha)\big). \qedhere $$ 
\end{proof}

We are finally ready to prove the special form of the projective bundle formula in the case of symplectic bundles.

\begin{proposition} \label{prop Projskew}
Let $E$ be a vector bundle of rank $2n$ over $X\in \SM$, endowed with an everywhere nondegenerate skew-symmetric form $\langle\ ,\ \rangle$. Consider the two $\Omega^*(X)$-algebra homomorphisms   
$$\varphi,\varphi':\Omega^*(X)[\xi]\rightarrow\Omega^*(\Proj (E))$$
which respectively map $\xi$ to $c_1(O(1))$ and $c_1(O(-1))$. Both $\varphi$ and $\varphi'$ are epimorphisms and furthermore 
$${\rm Ker}\,\varphi={\rm Ker }\,\varphi'=\Big(P_{2n}^E\big(\xi,\chi(\xi)\big),P_{2n-1}^E\big(\xi,\chi(\xi)\big) \Big).$$
\end{proposition}
\begin{proof}
We first take $\varphi$ into consideration. In view of theorem \ref{th Proj} we only need to prove the statement involving the kernel. For this we observe that, since $h_i(\xi,0)=\xi^i$, from the definition of $P^E_j$ together with (\ref{eq symm},ii) one can easily deduce 
$$P_{2n}^E\big(\xi,0\big)=P_{2n}^E\big(\xi,\chi(\xi)\big)-\chi(\xi)\cdot P_{2n-1}^E\big(\xi,\chi(\xi)\big)$$
and therefore we have one of the inclusions. We now need to show that both polynomials belong to the kernel. For this we apply proposition \ref{prop geom} with respect to the line bundle $L:=\text{Ker\,}(\pi^*E\srarrow O(1))^\perp$. By observing that the $P^E_i$'s are symmetric and that $L\simeq O(1)^\vee$,  we obtain
$$c_i(L^\perp/L)=P_i^E(c_1(L),c_1(L^\vee))=P_i^E(c_1(L^\vee),c_1(L))=\varphi(P_{i}^E(\xi,\chi(\xi))).$$
It now suffices to recall that $L^\perp/L$ has rank $2n-2$ and hence both $c_{2n}$ and $c_{2n-1}$ vanish.

Let us now move on to the second homomorphism, which we want to relate to the first one. For this purpose we consider $\psi$, the endomorphism of $\Omega^*(X)[\xi]$ which maps $\xi$ to $\chi(\xi)$, viewed as a polynomial. As a direct consequence of the defining property of $\chi$, we have that $\varphi=\varphi'\circ \psi$ and $\varphi'=\varphi\circ \psi$. While the first equality  proves the surjectivity of $\varphi'$, the second one allows us to relate the two kernels: $\text{Ker\,} \varphi\supseteq \psi(\text{Ker\,} \varphi')$.

    We are now ready to prove that the two kernels actually coincide. An easy calculation shows that both generators of $\text{Ker\,} \varphi$ are mapped to 0 by $\varphi'$ and hence $\text{Ker\,} \varphi'\supseteq \text{Ker\,} \varphi$. In fact, one only has to observe that, by definition, both polynomials are symmetric in the two entries and that $\chi(\chi(\xi))=\xi$, if one disregards the terms of degree higher than $\text{dim}_k X$. It should be noticed that, in view of the relation obtained earlier, we also proved that $\text{Ker\,}\varphi\supseteq\psi(\text{Ker\,}\varphi)$. We will now use both these observations to prove the missing inclusion. If $\alpha\in \text{Ker\,} \varphi'$, then both $\psi(\alpha)$ and $\psi^2(\alpha)$ belong to $\text{Ker\,} \varphi$ and this completes the proof, since $\varphi(\alpha)=\varphi(\psi^2(\alpha))$.  
\end{proof}

\section{Nested flag bundles and their fundamental classes}

The main goal of this section is to describe in detail the relationship between the algebraic cobordism rings of flag bundles of different sizes. We first illustrate how to deal with full flag bundles and then we continue by presenting the symplectic case, highlighting the differences between the two situations. 

\subsection{Full flag bundles}
In order to perform the actual computation of the fundamental classes of the full flag bundles, we need to be able to describe how the Schubert varieties of different ambient spaces are related. Before we recall the precise definition of this family of subvarieties, let us introduce some notation related to their indexing set. 

\paragraph{The symmetric group:}
 Let $S_n$ denote the $n$-th symmetric group and $s_i$ the elementary transposition $(i\ i+1)$. The length $l(w)$ of a permutation is defined as the minimal number of elementary transpositions needed in order to express $w$. In $S_n$ the maximum of the length function is achieved by  $w_0^{(n)}=(1\,n)(2\,n-1)\trecd (\lceil n/2\rceil\lceil (n+1)/2\rceil)$, which is usually referred to as the longest permutation.
 In order to relate symmetric groups of different size we consider the inclusion $e_{n-1}:S_{n-1}\irarrow S_n$  which maps $w$ to the permutation defined by setting $e_{n-1}(w)(n)=1$ and $e_{n-1}(w)(i)=w(i)+1$. By composing these inclusions we obtain an embedding $e_m$ from any smaller symmetric group $S_m$ into $S_n$. Notice that this embedding maps $w_0^{(m)}$ to $w_0^{(n)}$. Finally, we will write $\nu_m$ to denote $e_m(id)$.


\subsubsection*{Some identifications between Schubert varieties}
Let $V_n\rightarrow X$ be a vector bundle of rank $n$ over a scheme $X\in \SM$. Denote by $\flag V_n\stackrel{\pi_n}\longrightarrow X$ the full flag bundle associated to $V_n$, by $U\unddot^{(n)}=(0=U_0^{(n)}\subset U_1^{(n)}\subset \tred \subset U_n^{(n)}=\pi_n^*V_n)$ the universal flag of subbundles of $V_n$ and by $M_i^{(n)}$ the line bundles $U_{n+1-i}^{(n)}/U_{n-i}^{(n)}$ obtained via the filtration. Let us recall that $\flag V_n$ is characterised by a universal property, i.e. for every choice of a morphism $f:Y\rightarrow X$ and of a flag of subbundles $F\unddot$ of $f^* V_n$ there exists a unique $\widetilde{f}:Y\rightarrow \flag V_n$ such that $\pi_n\circ \widetilde{f}= f$ and $\widetilde{f}^* U\unddot^{(n)}=F\unddot$. 
It is woth recalling that $\flag V_n$ can be constructed explicitly as an iterated $\Proj^1$-bundle: one first considers $\Proj(V_n)\stackrel{\pi}\longrightarrow X$, which describes the hyperplane bundle of $V_n$, and then the construction is repeated with $\text{Ker}\,(\pi^*V_n \srarrow O(1))$ over $\Proj(V_n)$. The universal flag is assembled together by pulling back  the different kernels to the last projectivisation.  

Let us now recall the definition of 
 Schubert varieties. For this one requires the choice of a full flag, so let us fix  $W\unddot=(0=W_0\subset W_1\subset\tred \subset W_n=V)$ and set $L_i:=W_i/W_{i-1}$. This family of varieties is indexed by elements of $S_n$ and each of its elements is defined as 
\begin{align} \label{def Schubert}
\Omega^{(n)}_w:=\left\{x\in \flag V_n \ |\ {\rm dim}_k\left(\pi_n^*W_i(x)\cap U^{(n)}_j(x)\right)\geq r_w(i,j) \ \forall i,j \right\}\hspace{0.05 mm},
\end{align}
where $r_w(i,j):=|\{l > n-j \ | \ w(l)\leq i \}|$. A key observation, that will play an important role in our discussion, is that some of the elements of this family are isomorphic to full flag bundles of vector bundles of smaller rank. In order to apply inductive methods, we now want to describe explicitly how the different families of Schubert varieties are related. For this we set $V_m:=V_n/W_{n-m}$ and observe that, for each of these bundles, $W\unddot$ defines a full flag $W\unddot^{(m)}=(0\subset W_{n-m+1}/W_{n-m}\subset \trecd \subset V_n/W_{n-m}=V_m)$. As a consequence one has varieties $\{\Omega^{(m)}_{w}\}_{w\in S_m}$ lying inside of each $\flag V_m$. Furthermore, the remaining part of the flag $W\unddot$ can be used (once it is pulled-back to $\flag V_m$) to prolongue the universal flag $U\unddot^{(m)}$ and obtain, via the universal property of $\flag V_n$, a morphism $\iota_{m}:\flag V_m \rightarrow \flag V_n$. 

Let us provide some more details by looking at the case $m=n-1$. The universal flag $U^{(n-1)}\unddot$  of subbundles of $V_{n-1}$ induces a full flag of $\pi_{n-1}^*V_n$ in which the term representing the line bundle is $\pi^*_{n-1}W_1$ and therefore the universal property ensures that $\iota_{n-1}^* U_1^{(n)}=\pi_{n-1}^*W_1$. Similarly, for $1\leq j \leq n-m$ one has $\iota_m^*U_j^{(n)}=\pi^*_{m}W_j$ and as a consequence one has the following equalities: 
\begin{align} \label{eq pullbackA}
i) \ \iota_m^*M_j^{(n)}=M_j^{(m)} \quad \text{ for }\  1\leq j \leq m;
\quad ii) \ \iota_m^*M_{j}^{(n)}=\pi^*_m L_{n+1-j}  \quad \text{ for }\   m+1\leq j \leq n.
\end{align}

We now want to identify which Schubert varieties can also be viewed as flag bundles. By doing this we will in particular show that $\flag V_1\simeq \Omega_{w_0}^{(n)}$ and it is also isomorphic to the base scheme $X$.

\begin{lemma}\label{lemma Schubert flag}
For every $m < n$ the flag bundle $\flag V_m$ and $\Omega_{\nu_m}^{(n)}$ are isomorphic as schemes over $\flag V_n$.  
\end{lemma}

\begin{proof}
First of all let us recall that the scheme structure on Schubert varieties is given by defining them as the intersection of a certain family of zero schemes, each of which arises from one of the conditions in (\ref{def Schubert}). To be more precise, one translates the requirement on the dimension of the intersection into a condition on the rank of the morphism $f_{ij}:\pi_n^* W_i\rightarrow \pi^*_n V_n/ U^{(n)}_j$. The zero scheme is then obtained as the locus in which an appropriate exterior power of $f_{ij}$ vanishes.

 The important point for us is that it is possible to characterise $\Omega_{\nu_m}\irarrow \flag V_n$  as the fibre product of a diagram involving $\flag V_n$ and two bundle sections. One can check that if $\Omega_{\nu_m}$ is replaced with $\flag V_m \stackrel{\iota_m}\rightarrow \flag V_n$, then the same diagram still commutes and therefore by the universality we obtain a morphism $\flag V_m\rightarrow\Omega_{\nu_m}^{(n)} $. Conversely, it is a consequence of the defining conditions of $\Omega_{\nu_m}^{(n)}$ that the restriction of the universal flag $U\unddot^{(n)}$ defines a full flag of the pull-back of $V_m$, hence we also obtain a morphism in the opposite direction. To finish the proof it now suffices to verify that the two compositions fit in the diagrams describing the universal properties of $\Omega_{\nu_m}$ and $\flag V_m$. It then follows, in view of the uniqueness, that they have to be the identity.
\end{proof}

It follows from this lemma that as a closed subscheme $\flag V_m\stackrel{\iota_m}\irarrow \flag V_n$ is defined by the condition that the first $n-m$ subbundles have to coincide with those of the given flag $W\unddot$. This is just a particular instance of a more general phenomenon relating the two different families of Schubert varieties.
\begin{proposition}\label{prop induction geometry}
For every $w\in S_m$ the isomorphism of lemma \ref{lemma Schubert flag} induces the following isomorphism
$$\Omega^{(m)}_w\simeq \Omega^{(n)}_{e_m(w)}$$
of schemes over $\flag V_n$.
\end{proposition}
\begin{proof}
It is clear that it suffices to verify the claim for $m=n-1$, since the general statement easily follows by induction. 
By definition of $e_{n-1}$ for all $w$ we have $e_{n-1}(w)(n)=1$, therefore the Schubert varieties of the form $\Omega^{(n)}_{e_{n-1}(w)}$ have to satisfy the defining equations of $\Omega_{\nu_{n-1}}^{(n)}\simeq \flag V_{n-1}$. After imposing the equality of $U_1^{(n)}$ with $\pi_n^*W_1$, one has to rewrite the remaining conditions in terms of $U^{(n)}\unddot/U_1^{(n)}$ and $ W\unddot^{(n-1)}$ and these will give rise to a subscheme of $\flag V_{n-1}$. Notice that the restatements only involve the restriction of $e_{n-1}(w)$ to $\{1,\tred,n-1\}$ and that it is necessary to subtract 1 to all rank conditions, since $U_1^{(n)}$ has been removed. As a consequence one sees that the new equations are precisely the ones defining $\Omega_{w}^{(n-1)}$.
\end{proof}



\subsubsection*{The fundamental classes of $\flag V_m$}

Let us now illustrate the algebraic side of the picture determined by the inclusions $\iota_m$. To achieve this, we first recall the description of the ring structure of  $\Omega^*(\flag V_m)$.

\begin{proposition}[\protect{\cite[Theorem 2.6]{SchubertHornbostel}}] \label{prop Flag ring}
Let $E$ be a vector bundle of rank $n$ over $X\in\SM$ and let $J$ be the ideal of $\Omega^*(X)[t_1,\trecd,t_n]$ generated by the elements $c_i(E)-e_i(\textbf{t})$, where, for $1\leq i\leq n$, $c_i(E)$ is the $i$-th Chern class of $E$ and $e_i(\textbf{t})$ is the $i$-th elementary symmetric function. Then   
$$ \Omega^*(X)[\textbf{t}]/J\hspace{0.05 mm}\simeq\Omega^*(\flag E),$$
where the isomorphism maps the $t_i$'s to the Chern roots of $\pi_n^*E$. 
\end{proposition}

Let us now focus on the morphisms associated to the inclusions. It is easy to see that the pullback maps $\iota_m^*:\Omega^*(\flag V_n)\rightarrow \Omega^*(\flag V_m)$ are determined by (\ref{eq pullbackA}). More explicitly, if we set $c_1(M_j^{(m)})=x_j^{(m)}$ and $c_1(\pi_m^* L_j)=y_j$, we get 
$$i)\ \iota^*_m\big(x_j^{(n)}\big)=x_j^{(m)} \ \ \text{for } \ 1\leq j \leq m; 
\ ii)\ \iota^*_m\big(x_j^{(n)}\big)=y_{n+1-j}\ \ \text{for }\ m+1\leq j \leq n.$$
Notice that, in view of \textit{i)}, we can drop the superscript from $x_i$'s.

On the other hand, since the pullbacks $\iota^*_m$ are surjective, the description of the push-forward maps $\iota_{m*}$ is reduced to the identification of the image of the identity. For this we need the following lemma, which, through an inductive argument, will allow us to deal with the general case.
\begin{lemma}\label{lem push}
Let $\iota_j^{j+1}:\flag V_j\irarrow \flag V_{j+1}$ be the inclusion arising from the identification of $\flag V_j$ with $\Omega^{(j+1)}_{\nu_j}$. Then 
$$(\iota_j^{j+1})_*(1_{\flag V_j})=\prod_{i=1}^j F\Big(x_i,\chi(y_{n-j})\Big).$$
\end{lemma}
\begin{proof}
In view of lemma  \ref{lemma Schubert flag} one can identify $\flag V_j$ with the zero scheme of the section arising from the map $\pi_{j+1}^*W_1^{(j+1)}\rightarrow \pi^*_{j+1} V_{j+1}/U_1^{(j+1)}$. Since $\flag V_j$ is a smooth scheme and its codimension in $\flag V_{j+1}$ coincides with the rank of $(\pi^*_{j+1}W_1^{(j+1)})^\vee\otimes (\pi^*_{j+1}V_{j+1}/U_1^{(j+1)})$, its fundamental class in $\Omega^*(\flag V_{j+1})$ can be computed as the top Chern class of this vector bundle (see \cite[Lemma 2.1]{ThomHudson}). The right hand side of the statement is then obtained through a computation with Chern roots.
\end{proof}

We are now ready to prove a formula for the fundamental classes of the Schubert varieties~$\Omega^{(n)}_{\nu_m}$.
 
\begin{theorem} \label{thm flags class}
Let $V_n$ be a vector bundle of rank $n$ over $X\in \SM$ and $W\unddot$ a full flag of subbundles. Fix a positive integer $m\leq n$ and set $V_m:=V_n/W_{n-m}$. As an element of $\Omega^*(\flag V_n)$ the fundamental class of the Schubert variety $\Omega^{(n)}_{\nu_m}\simeq\flag V_m$ is given by
$$[\flag V_m]_{\Omega}=\prod_{\substack{i+j\leq n\\1\leq j\leq n-m}} F\Big(x_i, \chi(y_j)\Big),$$
where  $x_i:=c_1\big(U^{(n)}_{n+1-i}/U^{(n)}_{n-i}\big)$, $y_i:=c_1\big(W_i/W_{i-1}\big)$ and $U\unddot^{(n)}$ is the universal flag of subbundles of $V_n$.
Moreover, using the identification of proposition \ref{prop Flag ring}, this formula can be used to explicitly describe the push-forward morphism $\iota_{m*}$. One~has
\begin{align*}
\Omega^*(\flag V_m)&\stackrel{\iota_{m*}}\longrightarrow \Omega^*(\flag V_n)\\
P\quad &\mapsto P\cdot [\flag V_m]_{\Omega}
\end{align*}
where $P(x_1,\tred,x_m)$ is a polynomial with coefficients in $\Omega^*(X)$.
\end{theorem}

\begin{proof}
First of all let us observe that, in view of the surjectivity of $\iota_m^*$, the compatibility between the structures of $\Omega^*(\flag V_m)$ (viewed as a ring and as a $\Omega^*(\flag V_n)$-module) reduces the second claim to the first. 
In order to obtain the expression of the fundamental class of $\flag V_m$, we proceed by induction on $l=n-m$ through a repeated use of lemma \ref{lem push}. Since $\iota_{n-1*}(1_{\flag V_{n-1}})=[\flag V_{n-1}]_\Omega$, the base of the induction $l=1$ is precisely the lemma and by inductive hypothesis we can assume that 
$$[\flag V_{m+1}]_\Omega:=[\flag V_{m+1}\stackrel{\iota_{m+1}}\longrightarrow\flag V_n]_\Omega=\iota_{m+1*}(1_{\flag V_{m+1}})
=\prod_{\substack{i+j\leq n\\1\leq j\leq n-m-1}} F\Big(x_i, \chi(y_j)\Big).$$
It now suffices to notice that, in view of the second statement, we have the following chain of equalities
$$[\flag V_m]_\Omega=\Big(\iota_{m+1*}\circ(\iota_{m}^{m+1})_*\Big)(1_{\flag V_m})=(\iota_{m}^{m+1})_*(1_{\flag V_m})\cdot[\flag V_{m+1}]_\Omega = \prod_{i=1}^m F\Big(x_i,\chi(y_{n-m})\Big)\cdot[\flag V_{m+1}]_\Omega,$$
which yield the result once one combines together the two products.
\end{proof}

\begin{remark}
It is worth pointing out that the previous proposition recovers in a more systematic way the computation of  $[\Omega_{w_0}^{(n)}]_\Omega$ which appeared in \cite[Proposition 4.5]{ThomHudson}.  
\end{remark}


\begin{corollary}\label{cor Schubert classes}
Let $w\in S_n$ such that there exist $m<n$ and $w'\in S_m$ for which $e_m(w')=w$. Then in $\Omega^*(\flag V_n)$ the fundamental class of $\Omega^{(n)}_w$ can be written as
$$[\Omega^{(n)}_w]_\Omega=[\Omega^{(m)}_{w'}]_\Omega\cdot 
\prod_{\substack{i+j\leq n\\1\leq j\leq n-m}} F\Big(x_i, \chi(y_j)\Big),$$
where $[\Omega^{(m)}_{w'}]_\Omega$ is viewed as a polynomial in the Chern roots $x_1,\tred,x_m$ with coefficients in $\Omega^*(X)$.
\end{corollary}

\begin{proof}
Since, in view of proposition \ref{prop induction geometry}, the Schubert varieties $\Omega_w^{(n)}$ and $\Omega_{w'}^{(m)}$ are isomorphic as schemes over $\flag V_n$, we have 
$$[\Omega_w^{(n)}]_\Omega= [\Omega_w^{(n)}\irarrow \flag V_n]_\Omega=
\iota_{m*}\left( [\Omega_{w'}^{(m)}\irarrow \flag V_m]_\Omega\right)=
\iota_{m*}\left([\Omega_{w'}^{(m)}]_\Omega\right)=
[\Omega_{w'}^{(m)}]_\Omega \cdot \iota_{m*}(1_{\Omega^*(\flag V_m)}).$$
To complete the proof it now suffices to apply  theorem \ref{thm flags class}.
\end{proof}

\subsection{Symplectic flag bundles}
We now present the adaptation of the previous results to the case of bundles endowed with a nondegenerate skew-symmetric form. As before, we begin by introducing the indexing set.
\paragraph{The hyperoctahedral group:}
One important extension of the symmetric group is the hyperoctahedral group $\mathbf{W}_n$. It can be viewed as a subgroup of the permutations over the set $\{1,\overline{1},2,\overline{2},\tred,n,\overline{n}\}$, where $\bar{i}$ should be thought as $-i$. For this reason the elements of $\mathbf{W}_n$ are sometimes referred to as signed permutations.
 To be more specific, $\mathbf{W}_n$ consists of all permutations $w$ such that $w(\overline{i})=\overline{w(i)}$  for all $i\in\{1,\tred,n\}$.  Notice that in view of this defining property, every element is completely determined once one has set its restriction on $\{1,\tred,n\}$. For notational convenience we will also allow  $n$ to be zero 0, in which case we set $\mathbf{W}_0:=\{id\}$.

Exaclty as in the case of the symmetric group, one can define a notion of length for signed permutations $l^\mathbf{W}$ by considering their decompositions in terms of the elementary transpositions $s^\mathbf{W}_0=(1\overline{1})$ and $s^\mathbf{W}_1,\tred,s^\mathbf{W}_{n-1}$, where $s^\mathbf{W}_i=(i\ i+1)(\overline{i}\ \overline{i+1})$. In this case the longest element $w_0^{\mathbf{W}(n)}$ is $(1\ \overline{1})(2\ \overline{2})\trecd(n\ \overline{n})$. To every $l$-tuple of indices $0\leq i_j\leq n-1$ we associate $s_I^\mathbf{W}:=s_{i_1}^\mathbf{W}\trecd s_{i_l}^\mathbf{W}$ and we say that $I$ is a minimal decomposition of $s_I^\mathbf{W}$ if $l^\mathbf{W}(s_I^\mathbf{W})=l.$

 Finally, let us relate hyperoctahedral groups of different size. The image of $w$ under the embedding  $e^\mathbf{W}_{n-1}:\mathbf{W}_{n-1}\rightarrow \mathbf{W}_n$ is given by $e^\mathbf{W}_{n-1}(w)(i)=w(i)$ for $i\in \{1,\tred, n-1\}$ and by $e^\mathbf{W}_{n-1}(w)(n)=\overline{n}$. Through composition one obtains the inclusions $e_m^\mathbf{W}:\mathbf{W}_m\rightarrow \mathbf{W}_n$ and it is easy to check that also in this case $w_0^{\mathbf{W}(m)}$ is mapped to $w_0^{\mathbf{W}(n)}$. In line with the notation used for the symmetric group, $\nu^\text{\tiny $C$}_m$ will denote $e^\mathbf{W}_m(id)$.
In general, unless some confusion is likely to arise, we will drop the superscript $\mathbf{W}$ from the notation. 

The description of the hyperoctahedral group in terms of signed permutations proves to be the most suited when one has to deal with inductive arguments, however it makes the definition of Schubert varieties a little more involved, if compared with its interpretation as a subgroup of $S_{2n}$. In particular, to be able to express the rank conditions we will need the unique  bijection $g:\{1,\tred,2n\}\rightarrow\{1,\overline{1},\tred,n,\overline{n}\}$ which preserves the ordering induced by $\ZZ$ on both sets.

\subsubsection*{Other identifications between Schubert varieties}
Let $X\in\SM$. Given a vector bundle $V_n\rightarrow X$ of rank $2n$, endowed with an everywhere nondegenerate skew-symmetric form $\langle\ ,\, \rangle: V_n\otimes V_n\rightarrow \mathcal{O}_X$, it is possible to construct an analogue of the full flag bundle, which parametrises isotropic flags, i.e. flags of subbundles on which the form is identically 0. The construction of $\flagC V_n$, exactly as it happens in the type A case, is realised by means of a sequence of projective bundles.
 First of all one considers $\Proj (V_n)\stackrel{\pi}\longrightarrow X$ with its associated canonical quotient line bundle $\mathcal{O}(1)$. If the kernel of the projection $\pi^*V_n\rightarrow \mathcal{O}(1)$ is denoted $K$ and we set $U_1^{(n)}:=K^\perp$, then $\langle\ ,\, \rangle$ induces a non degenerate form on $K/U_1^{(n)}$, which has rank $2n-2$ and $U_1^{(n)}$, being a line bundle, is automatically isotropic since the form is skew-symmetric. It is therefore possible to repeat this procedure, assembling together all the isotropic line bundles that have been obtained into a full isotropic flag $U^{(n)}\unddot=(0=U_0^{(n)}\subset U_1^{(n)}\subset \trecd \subset U_n^{(n)})\irarrow \pi^*_n V_n$.
 These bundles are all defined over the last of the projective bundles, which will be denoted $\flagC V_n\stackrel{\pi_n}\longrightarrow X$. Also in this case $U^{(n)}\unddot$ represents the universal flag and the line bundles $U^{(n)}_{n+1-i}/U^{(n)}_{n-i}$ arising from the filtration will be denoted $M^{(n)}_i$. Notice furthermore that every full isotropic flag can be completed to a full flag by taking the orthogonal: for example in our case one sets $U^{(n)}_{n+i}:=U^{(n)\perp}_{n-i}$. It follows from the nondegeneracy of $\langle\ ,\, \rangle$ that the line bundles associated to the extension of the flag are dual to those of the original part, more precisely one has 
\begin{align}
U^{(n)}_{n+i}/U^{(n)}_{n-1+i}:=(M_i^{(n)})^\vee.
\end{align} 
It is worth mentioning that also $\flagC V_n$ is characterised by a universal property: for all morphisms $f:Y\rightarrow X$ and for every full isotropic flag $I\unddot$ of $f^*V_n$, there exists a unique morphism $\widetilde{f}:Y\rightarrow \flagC V_n$ such that $f=\pi_n\circ\widetilde{f}$ and $\widetilde{f}^*U\unddot^{(n)}=I\unddot$.

 Exactly as for full flag bundles, the choice of a full flag of isotropic subbundles $W\unddot=(0=W_0\subset W_1\subset \trecd \subset W_n)$ with associated line bundles $L_i:=W_i/W_{i-1}$ allows one to define Schubert varieties. For $w\in \mathbf{W}_n$ we set
$$\Omega^{(n)}_w:= \left\{x\in \flagC V\ |\ {\rm dim}_k\Big(\pi_n^*W_i(x)\cap U^{(n)}_j(x)\Big)\geq r_w\big(g(i),j\big) \ \text{for } (i,j)\in \{1,\tred,2n\}\times\{1, \tred, n\} \right\}.$$
Moreover, also in this case some of these subschemes are isomorphic to symplectic flag bundles of smaller rank and it is possible to identify their Schubert varieties with those of $\flagC V_n$. In this case one sets $V_m:=W^\perp_{n-m}/W_{n-m}$ and then it becomes possible to consider the Schubert varieties of $\flagC V_m$ associated to the induced flag $W\unddot/W_{n-m}$. Notice that it makes sense to include 0 among the possible values of $m$: in this case  $\flagC V_m$ will simply coincides with $X$.  

Thanks to its universal property we can again relate $\flagC V_n$ to the other flag bundles.
 In fact, the universal isotropic flag $U^{(m)}\unddot\irarrow \pi^*_m V_m$ can be pulled back to $\pi^*_m V_n$ and combined with $(\pi_m^* W_1\subset \trecd \subset \pi_m^* W_{n-m})$ to give rise to the full isotropic flag of $\pi^*_m V_n$ which yields $\iota_m^\text{\tiny $C$}:\flagC V_m\rightarrow \flagC V_n$. In terms of the line bundles arising from the given filtrations we have the following identifications:
\begin{align} \label{eq pullbackC}
i) \ \iota_m^{\text{\tiny $C$}*}M_j^{(n)}=M_j^{(m)} \quad \text{ for }\  1\leq j \leq m;
\quad ii) \ \iota_m^{\text{\tiny $C$}*}M_{j}^{(n)}=\pi^*_m L_{n+1-j}  \quad \text{ for }\   m+1\leq j \leq n.
\end{align} 



We conclude our description of the geometry of symplectic flag bundles by clarifying how the different families of Schubert varieties are related. This leads us to the analogues of  lemma \ref{lemma Schubert flag} and proposition \ref{prop induction geometry}.

\begin{lemma}\label{lemma Schubert flagC}
For every $m < n$ the flag bundle $\flagC V_m$ and $\Omega_{\nu_m^\text{\tiny $C$}}^{(n)}$ are isomorphic as  schemes over $\flagC V_n$.  
\end{lemma}

\begin{proof}
The proof essentially follows that of lemma \ref{lemma Schubert flag}.
\end{proof}

\begin{proposition}\label{prop induction geometryC}
For every $w\in \mathbf{W}_n$ the isomorphism of lemma \ref{lemma Schubert flagC} induces the following isomorphism
$$\Omega^{(m)}_w\simeq \Omega^{(n)}_{e_m(w)}$$
of schemes over $\flagC V_n$.
\end{proposition}

\begin{proof}
The main point is that all signed permutations belonging to the image of $e_{n-1}$ present the inversion $(n\, \overline{n})$ and as a consequence, in view of the lemma, each $\Omega_{e_{n-1}(w)}^{(n)}$ is a subscheme of $\flagC V_{n-1}$. At this point it suffices to reinterpret the remaining conditions, either by eliminating the unnecessary ones or by modifying the others accordingly, to observe that they match those of $\Omega_w^{(m)}$.   
\end{proof}

\subsubsection*{The fundamental classes of $\flagC V_m$}
In order to obtain a closed formula for $[\flagC V_m]_\Omega$,  we will make explicit the pushforward morphisms $\iota^\text{\tiny $C$}_{m*}$. To achieve this we first need a description of the ring structure of $\Omega^*(\flagC V_n)$.  


\begin{proposition} \label{prop FlagC ring}
Let $E_n$ be a vector bundle of rank $2n$ over $X\in\SM$ endowed with an everywhere nondegenerate skew-symmetric form $\langle\ ,\, \rangle: V_n\otimes V_n\rightarrow \mathcal{O}_X$. Let $J_n$ be the ideal of $\Omega^*(X)[\textbf{t}]$ generated by the elements $f_i^{(2n)}:=c_i(E_n)-e^{(2n)}_i\big(\textbf{t},\chi(\textbf{t})\big)$ where, for $1\leq i\leq 2n$, $c_i(E_n)$ is the $i$-th Chern class of $E_n$ and $e^{(2n)}_i\big(\textbf{t},\chi(\textbf{t})\big)$ is the $i$-th elementary symmetric function in $t_1,\tred,t_n$ and their formal inverses. Then   
$$ \Omega^*(X)[\textbf{t}]/J_n\hspace{0.05 mm}\simeq\Omega^*(\flagC E_n),$$
where the isomorphism maps the $t_i$'s to the Chern roots  of a maximal isotropic subbundle of $\pi_n^*E_n$. 
\end{proposition}
\begin{proof}
The statement can be either obtained from \cite[Corollary 5.2]{EquivariantKiritchenko} or through a direct computation based on the construction of $\flagC E_n$ as an iterated projective bundle. 

Let us first consider the case $n=1$, for which $\flagC E_1 \simeq \Proj (E_1)$. We can apply proposition \ref{prop Projskew} and reduce the statement to checking that $\left(P_{1}^{E_1}(t_n,\chi(t_n)),P_{2}^{E_1}(t_n,\chi(t_n))\right)=J_1$. To do this we observe that by definition $P_1^{E_1}(t_n,\chi(t_n))=f_1^{(2)}$ and it is straightfoward to verify that $P_2^{E_1}(t_n,\chi(t_n))=f_2^{(2)}-h_1\cdot f_1^{(2)}$. 
 
  For the inductive step one first considers $\Proj(E_n)$ and notices that the construction of the flag bundle ensures that $\flagC E_n$ is isomorphic to $\flagC T_{n-1}$, where $T_{n-1}$ is defined to be $U^{(n)}_{2n-1}/U_1^{(n)}$. From the inductive hypothesis it follows that for $I=\left(c_i(T_{n-1})-e_i^{(2n-2)}\big(\textit{\textbf{t}},\chi(\textit{\textbf{t}})\big)\right)_{1\leq i \leq 2n-2 }$ one has
$$ \Omega^*(\Proj (E_n))[t_1,\tred, t_{n-1}]/I\simeq\Omega^*(\flagC E_n).$$
We now use propositions \ref{prop Projskew} and  \ref{prop geom} to first rewrite the coefficient ring and then express the generators of $I$ as elements of $\Omega^*(X)[t_1,\tred,t_n]/(P_{2n-1}^{E_n},P_{2n}^{E_n})$.
As a result we are able to provide a description of the kernel of the morphism $\psi:\Omega^*(X)[t_1,\tred,t_n]\rightarrow \Omega^*(\flagC E_n)$ mapping $t_i$ to $c_1(M_i^{(n)})$. This ideal, denoted $K$, is generated by  $P_i^{E_n}-e_{i}^{(2n-2)}$ with $1\leq i \leq 2n$. 

We now need to prove that $K=J_n$. First of all let us note that, in view of the decomposition $c_u(E_n)=\prod_{i=1}^{n}c_u(M_i^{(n)})c_u(M_i^{(n)\vee})$, the generators of $J_n$ clearly belong to $K$, since they are mapped to zero by $\psi$. On the other hand, by subtracting 
$$e_i^{(2n-2)}=e_i^{(2n)}-
e_{i-1}^{(2n-2)}\big(t_n+\chi(t_n)\big)
-e_{i-2}^{(2n-2)} t_n \chi(t_n)$$
 from both sides of (\ref{eq rec}), one obtains
$$P^{E_n}_i-e_i^{(2n-2)}=c_i(E_n)-e_i^{(2n)}
-(P^{E_n}_{i-1}-e_{i-1}^{(2n-2)})\big(t_n+\chi(t_n)\big)
-(P^{E_n}_{i-2}-e_{i-2}^{(2n-2)})t_n \chi(t_n).$$
Provided that $e_0^{(m)}$ and $e_{-1}^{(m)}$ are respectively interpreted as 1 and 0, it is easy to see by induction that for all $i$ we have $P^{E_n}_i-e_i^{(2n-2)}\in J_n$. 
\end{proof} 

We now turn our attention to the morphisms between the cobordism rings of the different flag bundles.
The relations imposed by  (\ref{eq pullbackC}) allow us to completely describe the pullback maps $\iota_m^{\text{\tiny $C$}*}:\Omega^*(\flagC V_n)\rightarrow \Omega^*(\flagC V_m)$. In fact, if we set $c_1(M_j^{(m)})=x_j^{(m)}$ and $c_1(\pi_m^* L_j)=y_j$, we get 
$$i)\ \iota^{\text{\tiny $C$}*}_m(x_j^{(n)})=x_j^{(m)} \ \ \text{for } \ 1\leq j \leq m; \ ii)\ \iota^{\text{\tiny $C$}*}_m(x_j^{(n)})=\chi(y_{n+1-j})\ \ \text{for }\ m+1\leq j \leq n.$$
Again the superscript of the $x_i$'s can be dropped and, in view of the surjectivity of $\iota_m^{\text{\tiny $C$}*}$,  the push-forward maps $\iota^\text{\tiny $C$}_{m*}$ are completely determined by the image of the identity. 
\begin{lemma}
Let $\iota_{\ j}^{\text{\tiny $C$}\,j+1}:\flagC V_j\irarrow \flagC V_{j+1}$ be the inclusion arising from the identification of $\flagC V_j$ with $\Omega^{(j+1)}_{\nu^\text{\tiny $C$}_j}$. Then 
 $$(\iota_{\ j}^{\text{\tiny $C$}\,j+1})_*(1_{\flag V_j})=\prod_{i=1}^{j+1} F\Big(x_i,\chi(y_{n-j})\Big) \cdot\prod_{i=1}^j F\Big(\chi(x_i),\chi(y_{n-j})\Big)$$
\end{lemma}
\begin{proof}
 Except for the explicit computation of the top Chern class of  $W_1^\vee\otimes (\pi^*_{j+1}V_{j+1}/U_1^{(j+1)})$ the proof does not differ from that of lemma \ref{lem push}. 
\end{proof}

We are now in the position to obtain the analogues of theorem \ref{thm flags class} and corollary \ref{cor Schubert classes}. In both cases the structure of the proof is unchanged.

 
\begin{theorem} \label{thm flags class C}
Let $V_n$ be a vector bundle of rank $2n$ over $X\in \SM$ endowed with an everywhere nondegenerate skew-symmetric form and let $W\unddot$ be a full isotropic flag of subbundles. Fix a nonnegative integer $m\leq n$ and set $V_m:=W_{n-m}^\perp/W_{n-m}$. As an element of $\Omega^*(\flagC V_n)$ the fundamental class of the Schubert variety $\Omega^{(n)}_{\nu^\text{\tiny $C$}_m}\simeq\flagC V_m$ is given by
$$[\flagC V_m]_{\Omega}=\prod_{\substack{i+j\leq n\\1\leq j\leq n-m}} F\Big(x_i, \chi(y_j)\Big) \cdot
\prod_{\substack{i+j\leq n+1\\1\leq j\leq n-m}} F\Big(\chi(x_i), \chi(y_j)\Big),$$
where $x_i=c_1\big(U^{(n)}_i/U^{(n)}_{i-1}\big)$ and $y_i:=c_1\big(W_i/W_{i-1}\big)$.
Moreover, using the identification of proposition \ref{prop FlagC ring},, this formula can be used to explicitly describe the push-forward morphism $\iota^\text{\tiny $C$}_{m*}$. One~has
\begin{align*}
\Omega^*(\flagC V_m)&\stackrel{\iota^\text{\tiny $C$}_{m*}}\longrightarrow \Omega^*(\flagC V_n)\\
P\quad &\mapsto P\cdot [\flagC V_m]_{\Omega}
\end{align*}
where $P(x_1,\tred,x_m)$ is a polynomial with coefficients in $\Omega^*(X)$.
\end{theorem}

\begin{remark}
It should be noticed that in view of our conventions, which allow the case $m=0$, the previous theorem provides a formula for the fundamental class of the smallest Schubert variety~$\Omega_{w_0}^{(n)}$. 
\end{remark}

\begin{corollary} \label{cor Schubert class C}
Let $w\in \mathbf{W}_n$ such that there exist $m<n$ and $w'\in \mathbf{W}_m$ for which $e_m(w')=w$. Then in $\Omega^*(\flagC V_n)$ the fundamental class of $\Omega^{(n)}_w$ can be written as
$$[\Omega^{(n)}_w]_\Omega=[\Omega^{(m)}_{w'}]_\Omega\cdot
\prod_{\substack{i+j\leq n\\1\leq j\leq n-m}} F\Big(x_i, \chi(y_j)\Big) \cdot
\prod_{\substack{i+j\leq n+1\\1\leq j\leq n-m}} F\Big(\chi(x_i), \chi(y_j)\Big),$$
where $[\Omega^{(m)}_{w'}]_\Omega$ is viewed as a polynomial in $x_1,\tred,x_m$ with coefficients in $\Omega^*(X)$.
\end{corollary}

\section{Symplectic $CK^*$-Schubert classes via Bott-Samelson resolutions}
 
The main goal of this section is to provide a description of the push-forward classes of Bott-Samelson resolutions, a family of schemes over $\flagC V_n$ closely related to Schubert varieties, and derive some consequences at the level of connective $K$-theory. In order to do so, we first recall the definition of generalised divided difference operators associated to a formal group law, which are necessary to define the family of power series appearing in the formula. 

We first describe the situation for the universal formal group law $(\Laz, F)$. For $i\in\{0,\tred,n-1\}$ the generalised divided difference operators $A_i$ are defined on $\Laz[[x_1,\tred,x_n,y_1,\tred,y_n]]$ by the following formulas:

\begin{align} \label{eq divided}
A_0(f):=(1+\sigma_0)\frac{f}{F\Big(\chi(x_1),\chi(x_1)\Big)} \quad, \quad
A_i(f):=(1+\sigma_i)\frac{f}{F\Big(x_i,\chi(x_{i+1})\Big)}.
\end{align}
Here $1$ stands for the identity operator, while $\sigma_0$ maps $x_1$ to $\chi(x_1)$ and $\sigma_i$ exchanges $x_i$ and $x_{i+1}$. For brevity we will write $A_I$ for the composition $A_{i_l}\circ\trecd\circ A_{i_1}$, where $I=(i_1,\tred,i_l)$ and the  indices belong to $\{0,\tred,n-1\}$.

\begin{remark}
It worth noticing that, although the expressions in (\ref{eq divided}) do not a priori produce power series, one can show that for every monomial the numerator of the sum of the two fractions is actually divisible by the denominator and hence the operators are actually well defined. For a proof see \cite[Section 5]{SchubertHornbostel}.  
\end{remark}

For any other formal group law $(R,F_R)$ the operators $A_i^{(R,F_R)}$ can be obtained either by tensoring with respect to the classifying morphism $\varPhi_{(R,F_R)}$ or, equivalently, by replacing the Lazard ring and the univversal formal group law with the given ones. 

\begin{remark}
It is important to stress that in general these operators do not satisfy the braid relations. In particular this implies that it is not possible to associate an operator $A_w$ to any given $w\in \mathbf{W}_n$, since different minimal decompositions of $w$  may give rise to different operators.  
\end{remark}

With these operators at hand it is possible to introduce a symplectic analogue of the power series used in \cite{ThomHudson} to describe the push-forward classes of Bott-Samelson resolutions. 
\begin{definition}\label{def power}
Fix $n\in\NN$. To every tuple $I=(i_1,\tred,i_l)$ with $i_j\in\{0,\tred,n-1\}$ we associate the power series $\mathcal{S}\mathfrak{B}^{(n)}_I\in \Laz[[x_1,\tred,x_n,y_1,\tred,y_n]]$ by setting
$$\mathcal{S}\mathfrak{B}^{(n)}_\emptyset=
\prod_{i+j\leq n} F\big( x_i, y_j\big)  \cdot \prod_{i+j\leq n+1} F\big( \chi(x_i), y_j\big)
\quad \text{and}
\quad \mathcal{S}\mathfrak{B}^{(n)}_I= A_I \left(\mathcal{S}\mathfrak{B}^{(n)}_\emptyset\right)
\text{\ for }I\neq \emptyset. $$
In a similar fashion we define $\mathcal{S}\mathfrak{B}^{(A,n)}_I\in A^*(k)[[x_1,\tred,x_n,y_1,\tred,y_n]]$ for any given OCT $A^*$. 
\end{definition}
     
We now briefly recall the definition of Bott-Samelson resolutions in the symplectic case and for this we have to introduce the partial flags $\flagC_{\widehat{j}} V_n$. Each  of these schemes parametrises the isotropic flags $F\unddot \irarrow V_n$ in which only the $(n-j)$-th level is missing and it is easy to see that $\varphi_j:\flagC V_n\rightarrow \flagC_{\widehat j} V_n$ is actually a $\Proj^1$-bundle.
Bott-Samelson resolutions, denoted $R_I\stackrel{r_I}\rightarrow \flagC V_n$, are indexed by $l$-tuples of indices $I=(i_1,\tred,i_l)$ with $i_j\in\{0,\tred, n-1\}$. Their definition is given recursively and it is based on the observation that the Schubert variety $\Omega_{w_0}\irarrow \flagC V_n$ is always isomoprhic to the base scheme $X$ and hence, according to our conventions, smooth. 
As a consequence one sets $R_\emptyset=\Omega_{w_0}$ with $r_\emptyset$ being the inclusion into $\flagC V_n$. 
For any tuple $I$ of length at least 1 we can write $I=(I',j)$, consider the following fibre product
\begin{eqnarray}\label{diagram}
\xymatrix{
  R_{I'}\times_{\flagC_{\widehat{j}}V_n} \flagC V_n \ar[rr]^{pr_2} \ar[d]_{pr_1}& &\flagC V_n\ar[d]^{\varphi_j} \\
  R_{I'}\ar[r]^{r_{I'}}& \flagC V_n \ar[r]^{\varphi_j} &\flagC_{\widehat{j}}V_n    
}
\end{eqnarray}
and define $R_I:=R_{I'}\times_{\flagC_{\widehat{j}}V_n} \flagC V_n$ and $r_I:=pr_2$. It immediately follows from the construction that $R_I$ is smooth since both $pr_1$ and $R_{I'}$ are.

The importance of Bott-Samelson resolutions comes from their relationship with Schubert varieties: each $\Omega_w$ is birational to at least one member of this family. As a matter of fact, even though they are not always mentioned explicitly, it is for this reason that Bott-Samelson resolutions appear in essentially all the computations of Schubert classes in $CH^*$ and $K^0$. The next proposition clarifies the nature of this relationship, as well as stating some related properties of Schubert varieties.

\begin{proposition}\label{prop resolution}
Let $I=(i_1,\tred, i_l)$ be a minimal decomposition of $s_I$ and set $w=w_0 s_I$. Then 

1) $r_I(R_I)=\Omega_{w}$ and  the resulting map $R_I\rightarrow \Omega_{w}$ is a projective birational morphism. In particular if $X\in \SM$, then $R_I$ is a resolution of singularities of $\Omega_{w}$;

2) i) $r_{I*}\mathcal{O}_{R_I}=\mathcal{O}_{\Omega_w}$  as coherent sheaves and therefore $\Omega_w$ is a normal scheme;

\ \ \ ii) $R^q r_{I*}\mathcal{O}_{R_I}=0$ for q>0, hence $\Omega_w$ has at worst rational singularities. 
\end{proposition}

\begin{proof}
For part (1) see \cite[Appendix C]{SchubertFulton}. For part (2) see \cite[Theorem 4]{SchubertRamanathan}.
\end{proof}

Another important ingredient for the computation of Schubert classes is represented by the morphisms $\varphi_i^*\varphi_{i*}$, which have to be described explicitly. In fact one discovers (see \cite[Theorem 5.30]{SymmetricVishik} for the most general statement involving push-forward along $\Proj^n$-bundles and \cite[Section 2.1]{SchubertHornbostel} for a more specific treatment) that the algebraic counterpart of these push-pull morphisms is given by the generalised divided difference operators. More precisely, if $\overline{A_i}$ denotes the morphism induced on $\Omega^*(\flagC V_n)$ through the identification of proposition \ref{prop FlagC ring}, one gets
\begin{align} \label{eq divided}
\overline{A_i}=\varphi_i^*\varphi_{i*}
\end{align}
for any $i\in\{0,\tred,n-1\}$. With this description at hand we are now able to express the classes $\mathcal{R}^A_I:=r_{I*}[R_I]_A$ as elements of $A^*(\flagC V_n)$ for any choice of $A^*$. This establishes an exact analogue of the formula given in \cite[Theorem 4.8, Corollary 4.9]{ThomHudson} for full flag bundles.
  
\begin{theorem}\label{thm Bott}
Given $X\in\SM$, let $V_n$ be a vector bundle of rank $2n$  endowed with an everywhere nondegenerate skew-symmetric form and let $W\unddot$ be a full isotropic flag of subbundles of $V_n$. Set $M^{(n)}_i:=U^{(n)}_{n+1-i}/U^{(n)}_{n-i}$ and $L_i:=W_i/W_{i-1}$, where $U^{(n)}\unddot$ is the universal isotropic flag over $\flagC V_n\stackrel{\pi_n}\longrightarrow X$. For any tuple $I=(i_1,...,i_l)$ with $i_j\in\{0,1,\tred,n-1 \}$ let us consider the associated Bott-Samelson resolution $R_I\stackrel{r_I}\rightarrow \flagC V_n$. For any OCT $A^*$, its pushforward class in $A^*(\flagC V_n)$ is
given by
$$\mathcal{R}^A_I=\mathcal{S}\mathfrak{B}^{(A,n)}_I \left(c_1\big(M^{(n)}_k\big),c_1(L_j^\vee)\right)=
\mathcal{S}\mathfrak{B}^{(A,n)}_I \left(c_1\big(M^{(n)}_k\big),\chi \big(c_1(L_j)\big)\right).$$
\end{theorem}

\begin{proof}
First of all, let us observe that the statement of the theorem follows from the universal case $A^*=\Omega^*$ by applying the canonical morphism $\vartheta_{A^*}$. The proof is by induction on the length of the tuple $I$ and the base of the induction is given by theorem \ref{thm flags class C}. In fact, for $m=0$ one obtains an expression for $[\Omega_{w_0}^{(n)}]=\mathcal{R}_\emptyset$, which, by construction, coincides with $\mathcal{S}\mathfrak{B}^{(n)}_\emptyset$ when one substitutes the given roots.
For the inductive step we can write $I=(I',j)$ and, in view of the transversality of (\ref{diagram}), of the inductive hypothesis and of (\ref{eq divided}), we get 
$$\mathcal{R}^A_I=r_{I*}[R_I]_\Omega=r_{I*}\big(pr_1^*[R_{I'}]_\Omega\big)=
\varphi_{j*}(\varphi_j^*\mathcal{R}^A_{I'})=
 \overline{A_j}\Big(\mathcal{S}\mathfrak{B}^{(n)}_{I'}\left(c_1\big(M^{(n)}_k\big),c_1(L_j^\vee)\right)\Big). 
$$  
 To finish the proof it is enough to recall definition \ref{def power} and then observe that $\overline{A_j}\circ\overline{A_{I'}}=\overline{A_I}$.
\end{proof}

We now want to specialise our result to the case of connective $K$-theory. This choice is motivated by two  observations. Firstly, as pointed out in section \ref{section prem}, $CK^*$ has a well-defined notion of fundamental class, which allows one to pose the question of how to express Schubert classes.  Secondly, this theory represents the most general setting in which divided difference operators satisfy the braid relations and, as a consequence, the series $\mathcal{S}\mathfrak{B}_I^{(CK,n)}$ can be organised using  $\mathbf{W}_n$ as an indexing set, as opposed to tuples. More specifically, the analogy with the situation regarding full flag bundles (see \cite[Remark 4.7]{ThomHudson}) suggests to consider a naive analogue of the double $\beta$-polynomials (defined by Fomin and Kirillov in \cite{GrothendieckFomin,DoubleGrothendieckFomin}) and to set  
\begin{align} \label{eq beta}
\mathcal{S}\mathfrak{H}_w^{(n)}:= \mathcal{S}\mathfrak{B}_{I(w)}^{(CK,n)},
\end{align}
where $I(w)$ is any minimal decomposition of $w_0w$. In this case the divided difference operators $\phi_i:=\vartheta_{CK}(A_i)$ and the polynomial associated to the longest element can be described more explicitly:
$$i)\ \phi_0(P)=\frac{P(\ominus x_1,x2,\dots,x_n)-(1-\beta x_1)^2 P}{2x_1-\beta x_1^2}
\quad,\quad
\phi_i(P)=\frac{(1-\beta x_i)P-(1-\beta x_{i+1})\sigma_i(P)}{x_i-x_{i+1}}\,;$$
$$
ii)\ \mathcal{S}\mathfrak{H}_{w_0^{(n)}}^{(n)}=\prod_{i+j\leq n} x_i\oplus y_j  \cdot
 \prod_{i+j\leq n+1} \ominus x_i \oplus y_j.
$$

We are now ready to provide a description of the Schubert classes for connective $K$-theory.
\begin{theorem}
 Let $V_n\rightarrow X$ be a vector bundle of rank $2n$ endowed with an everywhere nondegenerate skew-symmetric form, with $X\in\SM$. Let $W\unddot$ be a full isotropic flag of subbundles of $V_n$ and $\{\Omega_w\}_{w\in \mathbf{W}_n}$ the associated Schubert varieties of $\flagC V_n\stackrel{\pi_n}\longrightarrow X$.  Set $M^{(n)}_i:=U^{(n)}_{n+1-i}/U^{(n)}_{n-i}$ and $L_i:=W_i/W_{i-1}$, where $U^{(n)}\unddot$ is the universal isotropic flag, then for any $w\in \mathbf{W}_n$ we have 
$$[\Omega_w^{(n)}]_{CK}=\mathcal{S}\mathfrak{H}_w^{(n)}\Big(c_1(M_i^{(n)}),c_1(L_j^\vee)\Big).$$ 
\end{theorem}

\begin{proof}
If we take $I$ to be a minimal decomposition of $w_0w$, then by proposition \ref{prop resolution} we know that $R_I$ is a resolution of singularities of $\Omega_w$. Since we also know that Schubert varieties have at worst rational singularities, we are in the position to apply lemma \ref{lem class}, which allows us to conclude that $[\Omega^{(n)}_w]_{CK}=\mathcal{R}_I^{CK}$ . Finally, one applies theorem \ref{thm Bott} and, in view of (\ref{eq beta}), replaces  $\mathcal{R}_I^{CK}$ with the right hand side of the statement.
\end{proof}

We finish by showing how the previous statement implies a description of the Schubert classes for both $CH^*$ and $K^0$. For this we first set some notation. Let us recall that by definition $\mathcal{S}\mathfrak{H}_w^{(n)}$ has coefficients in $CK^*(k)$ and, as a consequence, it depends on  $\beta$. We define  $\mathcal{S}\mathfrak{S}_w^{(n)}$ and $\mathcal{S}\mathfrak{G}_w^{(n)}$ to be the power series which are obtained by setting $\beta$ equal to 0 and $-1$. 

\begin{corollary}
Under the hypothesis of the preceding theorem we obtain the following equalities in $CH^*(\flagC V_n)$ and $K^0(\flagC V_n)$. 
$$i)\ [\Omega_w^{(n)}]_{CH}=\mathcal{S}\mathfrak{S}_w^{(n)}\Big(c_1(M_i^{(n)}),c_1(L_j^\vee)\Big)\quad;\quad ii) \ 
[\mathcal{O}_{\Omega_w^{(n)}}]_{K^0}=\mathcal{S}\mathfrak{G}_w^{(n)}\Big(c_1(M_i^{(n)}),c_1(L_j^\vee)\Big).$$ 
\end{corollary}

\begin{remark}
It should be noticed that, unlike what happens with the usual double $\beta$-polynomials, for $\mathcal{S}\mathfrak{H}_w^{(n)}$ one does not have a uniform description which is independent of $n$. Instances of this phenomenon are easy to find: for example one has  
$\mathcal{S}\mathfrak{S}_{(\overline{1},1)}^{(1)}=-x_1+y_1$ while $\mathcal{S}\mathfrak{S}_{(\overline{1},1)}^{(2)}=-x_1-x_2+y_1+y_2$.
\end{remark}


\bibliographystyle{siam}
\bibliography{biblio}

The Center for Geometry and its Applications (GAIA), POSTECH, Pohang city, Gyeong-sang-bugdo 790-784 South Korea
\vspace*{1\baselineskip}

\noindent \textit{E-mail address}: hudsont@postech.ac.kr

\end{document}